\documentclass{amsart}
\usepackage{amsmath,amsfonts,amssymb}

\theoremstyle{plain}

\newtheorem{theorem}{Theorem}
\newtheorem{lemma}{Lemma}
\newtheorem{remark}{Remark}
\newtheorem{corollary}{Corollary}

\author{Istv\'an Blahota}

\address{Institute of Mathematics and Computer Sciences\\
	University of Ny\'\i \-regyh\'aza\\
	H-4400 Ny\'\i regyh\'aza, P.O. Box 166\\
	Hungary}
\email{blahota.istvan@nye.hu}

\author{D\'ora Nagy}

\address{Institute of Mathematics and Computer Sciences\\
	University of Ny\'\i \-regyh\'aza\\
	H-4400 Ny\'\i regyh\'aza, P.O. Box 166\\
	Hungary}
\email{nagy.dora@nye.hu}

\title[Approximation by matrix transform means\dots]{Approximation by matrix transform means with respect to the Walsh system in Lebesgue spaces}

\keywords{Fourier series, Walsh-Paley system, rate of approximation, matrix transform means}

\subjclass{42C10}

\begin{document}

\begin{abstract} In this paper, we improve, complement and generalize (from N\"or\-lund to matrix transform means) a result of M\'oricz and Siddiqi \cite{MS} and some statements of Areshidze and Tephnadze \cite{AT}, and (from $T$ (weighted) to matrix transform means) Anakidze, Areshidze, Persson and Tephnadze \cite{AAPT}.
\end{abstract}

\maketitle

\section{Notation and definitions}

Let $\mathbb{P}$ be the set of positive natural numbers and $\mathbb{N}:=\mathbb{P}\cup \{0\}$.
Let denote the discrete cyclic group of order
$2$ by $\mathbb{Z}_2$. The group operation is the modulo $2$ addition. Let every subset
be open. The normalized Haar measure $\mu$ on $\mathbb{Z}_{2}$ is given
in the way that $\mu (\{ 0\})=\mu (\{ 1\})=1/2$. That is, the measure of a singleton is $1/2$. $G :=
\overset{\infty}{\underset{k=0}{{\times}}} \mathbb{Z}_{2},$
$G$
is called the dyadic group. The elements of dyadic group $G$ are the $0,1$ sequences. That is,
$x=(x_0,x_1,\dots,x_k,\dots)$ with $x_k\in \{0,1\}\ (k\in \mathbb{N}).$

The group operation on $G$ is the coordinate-wise addition (denoted by $+$),
the normalized Haar measure  $\mu $ is the product measure and the topology is the product topology. For an other topology on the dyadic group see e.g. \cite{BSU}.

Dyadic intervals are defined in the usual way
\[
I_0(x):=G,\ \ I_n(x):=\{y\in G : y=(x_0,\dots,x_{n-1},y_n,y_{n+1},\dots)\}
\]
for $x\in G, n\in\mathbb{P}$ and denote $I_{n}:=I_{n}(0)$. Intervals form a base for the neighbourhoods of $G$. 

Let $L_{p}(G)$ denote the usual Lebesgue spaces on $G$ (with the
corresponding norm $\Vert .\Vert_p$).

For the sake of brevity in notation, we agree to write $L_{\infty}(G)$ instead of $C(G)$ (the space of continuous functions on $G$) and set
$\Vert f\Vert_\infty:=\sup\{ |f(x)|: x\in G\}.$ Of course, it is clear that the space $L_{\infty}$ is not the same as the space of continuous functions, i.e. it is a proper subspace of it. But since in the case of continuous functions the supremum norm and the $L_{\infty}$ norm are the same, for convenience we hope the reader will be able to tolerate this simplification in notation.

Now, we introduce some concepts of Walsh-Fourier analysis.
The Rade\-macher functions are defined as
\[
r_n(x):=(-1)^{x_n} \quad
(x\in G, n\in\mathbb{N}).
\]
The sequence of the Walsh-Paley functions is the product system of the Rade\-macher functions. Namely,
every natural number $n$ can uniquely be expressed in the number system based $2$, in the form
\[
n=\sum_{k=0}^\infty n_k2^k,\quad  n_k\in \{ 0,1\} \ (k\in\mathbb{N}),
\] where only a finite number of $n_k$'s different from zero. 

Let the order of $n\in\mathbb{P}$ be denoted by $\vert n\vert :=\max \{ j\in\mathbb{N}: n_j\neq 0\}.$  It means $2^{|n|}\leq n< 2^{|n|+1}$. The Walsh-Paley functions are $w_0(x):=1$ and for $n\in\mathbb{P}$
\[
w_n(x):=\prod_{k=0}^{\infty}r_k^{n_k}(x)
=(-1)^{\sum_{k=0}^{|n|}n_{k}x_{k}}.
\]
Let ${\mathcal P}_n$ be the collection of Walsh  polynomials of order less than $n$, that is, functions of the form
\[
P(x)=\sum_{k=0}^{n-1}a_kw_k(x),
\]
where $n\in\mathbb{P}$ and $\{a_k\}$ is a sequence of complex numbers.

It is known \cite{hew} that the Walsh-Paley system $(w_n, n\in\mathbb{N})$ is the character system of $(G,+)$.

The $L_{p}(G)$ modulus of continuity are defined by
\[
\omega_p(f,\delta):=\sup_{|t|<\delta}\left\Vert f(.+t)-f(.)\right\Vert_p,
\]
for $f\in L_{p}(G)$, where $\delta>0$ with the notation
\[
\lvert x\rvert:=\sum_{i=0}^\infty \frac{x_i}{2^{i+1}} \quad \textrm{for all }x\in G.
\]
In the case $f\in C(G)$ we change $p$ by $\infty$.

The $j$th Fourier-coefficient, the $k$th partial sum of the Fourier series and the $n$th Dirichlet kernel is defined by
\[
	\hat f(j):=\int_{G} f(x)w_{j}(x)d\mu(x),\]
\[
S_{k}(f):=\sum_{j=0}^{k-1}\hat f(j)w_{j},\]
\[D_{n}:=
\sum_{k=0}^{n-1}w_{k},\ D_{0}:=0.
\]

Fej\'er kernels are defined as the arithmetical means of Dirichlet kernels and Fej\'er means are defined as the arithmetical means of partial sums of the Fourier series, that is,
\[
K_{n}:=\frac{1}{n}\sum_{k=1}^{n}D_{k}
\]
and
\[
\sigma_{n}(f):=\frac{1}{n}\sum_{k=1}^{n}S_{k}(f).
\]

Let $\{ q_k:k\in\mathbb{N}\}$ be a sequence of non-negative numbers. The $n$th N\"orlund mean of the Walsh-Fourier series is  defined by
\[
t_{n}(f):=\frac{1}{Q_n}\sum_{k=1}^{n}q_{n-k}S_k(f),
\]
where $Q_n:=\sum_{k=0}^{n-1}q_k$ $(n\in\mathbb{P})$. It is  always assumed that $q_0>0$ and 
\begin{equation}\label{Nor}
	\lim_{n\to\infty}Q_n=\infty.
\end{equation}
In this case, the summability method generated by $\{q_{k}\}$ is regular if and only if
\begin{equation}\label{reg}
	\lim_{n\to\infty}\frac{q_{n-1}}{Q_{n}}=0.
\end{equation}
As to this notion and result, we refer the reader to \cite{Moore}, pp. 37-38.

If $q_{k}=1/(k+1)$, we get the so-called N\"orlund logarithmic mean of the function $f$.

Let $\{ p_k: k\in\mathbb{P}\}$ be a sequence of nonnegative numbers. The $n$th $T$ (weighted) mean of Vilenkin-Fourier series is defined by
\[
	T_{n}(f):=\frac{1}{P_n}\sum_{k=1}^{n}p_{k}S_k(f),
\]
where $P_n:=\sum_{k=1}^{n}p_k\ (n\in\mathbb{P}).$ It is always assumed that $p_1>0$ and 
\[
\lim_{n\to\infty}P_n=\infty,
\]
which is the condition for regularity. 

In particular case N\"orlund and $T$ (weighted) means are Fej\'er ones (for all $k$ set $p_k:=1$ and $q_k:=1$).

Let $T:=\left( t_{i,j}\right)_{i,j=0}^{\infty}$ be a doubly infinite matrix of numbers. It is always supposed that matrix $T$ is triangular.
Let us define the $n$th decreasing diagonal  matrix transform mean determined by the matrix $T$
\[
\sigma_{n}^{T}(f):=\sum_{k=1}^{n}t_{k,n}S_{k}(f),
\]
where $\{ t_{k,n}: 1\leq k\leq n, \ k\in\mathbb{P}\}$ be a finite sequence of non-negative numbers for each $n\in\mathbb{P}$. 

With assumption $\sum_{k=1}^{n}t_{k,n}=1$, N\"orlund means are special matrix transform means, choosing 
\begin{equation}\label{conn}
	t_{k,n}:=\frac{q_{n-k}}{Q_{n}}.
\end{equation}
The $n$th matrix transform kernel is defined by
\[
K_{n}^{T}(x):=\sum_{k=1}^{n}t_{k,n}D_{k}(x).
\]
It is easily seen that
\[
\sigma_{n}^{T}(f;x)=\int_{G}f(u+x)K_{n}^{T}(u)d\mu(u),
\]
where $x,u\in G$. This equality (and its analogous versions for special means) shows us the necessity of observing kernel functions.

We introduce the notation
$\Delta t_{k,n}:=t_{k,n}-t_{k+1,n},$ where $k\in\{1,\ldots,n\},\ n\in\mathbb{P}$ and $t_{n+1,n}:=0$.	

\section{Historical overview}

Matrix transform means are common generalizations of several well-known summation methods. It follows by simple consideration that the N\"orlund means, the $T$ (weighted) means, the  Fej\'er (or the $(C,1)$) and the $(C,\alpha)$ means are special cases of the matrix transform summation method introduced above.

Our paper is motivated by works of M\'oricz and Siddiqi \cite{MS}, Areshidze and Tephnadze \cite{AT} on the Walsh-N\"orlund, and Anakidze, Areshidze, Persson and Tephnadze \cite{AAPT} on the $T$ (weighted) summation method. (For Vilenkin versions see \cite{AAB} and \cite{BN-Publ}.) As special cases, M\'oricz and Siddiqi obtained the earlier results given by Yano \cite{Y2}, Jastrebova \cite{J} and Skvortsov \cite{Skc} on the rate of the approximation by Ces\`aro means. 

The approximation properties of the Walsh-Ces\`aro means of negative order were studied by Goginava \cite{gog1}, the Vilenkin case was investigated by Shavardenidze \cite{Sh} and Tepnadze \cite{TT}. 

Fridli, Manchanda and Siddiqi  generalized the result of M\'oricz and Siddiqi for homogeneous Banach spaces and dyadic Hardy spaces \cite{FMS}. Recently, L. Baramidze, D. Baramidze, Memi\'c, Persson, Tephnadze and Wall presented  some results with respect to this topic \cite{T1,BPST,T3,T2,Mem}. See \cite{TE, W1}, as well. For two-dimensional results see \cite{BN3,BNT,N2,N1}.

M\'oricz and Siddiqi in \cite{MS} proved for N\"orlund means the following.

\begin{theorem}[M\'oricz and Siddiqi \cite{MS}, Theorem 1]\label{MS}
	Let $f\in L_{p}(G),\ 1\leq p\leq\infty$ and let $\{q_{k}: k\in\mathbb{N}\}$ be a sequence of nonnegative numbers such that
	\begin{equation}\label{MS1}
		\frac{n^{\gamma-1}}{Q^{\gamma}_{n}}\sum_{k=0}^{n-1}q_{k}^{\gamma}=O(1)\textrm{ for some } 1<\gamma\leq 2.
	\end{equation}
	If ${q_{k}}$ is non-decreasing, then
	\begin{equation}\label{MS2}
	 \left\|t_{n}(f)-f\right\|_{p}\leq \frac{5}{2Q_{n}}\sum_{j=0}^{|n|-1}2^{j}q_{n-2^{j}}\omega_{p}\left(f,\frac{1}{2^{j}}\right)+c\omega_{p}\left(f,\frac{1}{2^{|n|}}\right),
	\end{equation}
	while if ${q_{k}}$ is non-increasing, then
	\begin{equation}\label{MS3}
		\left\|t_{n}(f)-f\right\|_{p}\leq \frac{5}{2Q_{n}}\sum_{j=0}^{|n|-1}(Q_{n-2^{j}-1}-Q_{n-2^{j+1}-1})\omega_{p}\left(f,\frac{1}{2^{j}}\right)+c\omega_{p}\left(f,\frac{1}{2^{|n|}}\right).
	\end{equation}
\end{theorem}
\begin{remark}
	Condition \eqref{MS1} implies \eqref{Nor} and \eqref{reg}.
\end{remark}
Blahota and K. Nagy \cite{BN1} observed similar inequalities for matrix transform means. With our notations:
\begin{theorem}[Blahota and K. Nagy \cite{BN1}, Theorem 1]\label{BNthm}
	Let $f\in L_{p}(G), 1 \leq p \leq \infty$. For every $n\in\mathbb{N},\ \{t_{k,n}: 1\leq k\leq n\}$ be a finite sequence of non-negative numbers such that 
	\[
	\sum_{k=1}^n t_{k,n}=1
	\]
	is satisfied. 
	
	a) If the finite sequence  $ \{ t_{k,n}: 1\leq k \leq n\}$ is non-decreasing for a fixed $n$ and the condition 
	\begin{equation}\label{cond-1}
		t_{n,n}=O\left(\frac{1}{n}\right)
	\end{equation}
	is satisfied, then
	\begin{eqnarray*}\label{elso}
		\left\Vert\sigma^{T}_{n}(f)-f\right\Vert_{p}\leq 5\sum_{j=0}^{|n|-1}2^{j}t_{2^{j+1}-1,n}
		\omega_{p}\left(f,\frac{1}{2^{j}}\right)+c
		\omega_{p}\left(f,\frac{1}{2^{|n|}}\right)
	\end{eqnarray*}
	holds.
	
	b) If the finite sequence $ \{ t_{k,n}: 1\leq k \leq n \}$  is  non-increasing for a  fixed $n$, then 
	\begin{equation*}\label{masodik}
		\left\Vert\sigma^{T}_{n}(f)-f\right\Vert_{p}\leq 5\sum_{j=0}^{|n|-1}2^{j}t_{2^{j},n}
		\omega_{p}\left(f,\frac{1}{2^{j}}\right)+c
		\omega_{p}\left(f,\frac{1}{2^{|n|}}\right)
	\end{equation*}
	holds.
\end{theorem}

Iofina and Volosivets published similar results on Vilenkin systems with similar assumptions using different methods (independently form technics of M\'oricz, Rhoades, Siddiqi, Fridli and others) with respect to matrix transform means in \cite{IS}.

Areshidze and Tephnadze \cite{AT} obtained Theorem \ref{ATT}. They proved (in the non-decreasing case) Theorem \ref{MS}, with concrete coefficient and without condition \eqref{MS1}.

\begin{theorem}[Areshidze and Tephnadze \cite{AT}, Theorem 1]\label{ATT}
	Let $f\in L_{p}(G),\ 1\leq p<\infty$ and let $t_{n}$ be a regular N\"orlund mean generated by non-decreasing sequence $\{q_{k}: k\in\mathbb{N}\}$. Then
	\begin{equation}\label{AT}
		\left\|t_{n}(f)-f\right\|_{p}\leq 18\sum_{k=0}^{|n|-1}2^{k}\frac{q_{n-2^{k}}}{Q_{n}}\omega_{p}\left(f,\frac{1}{2^{k}}\right)+12\omega_{p}\left(f,\frac{1}{2^{|n|}}\right).
	\end{equation}
\end{theorem}
For similar results on $T$ (weighted) means see paper of Anakidze, Areshidze, Persson and Tephnadze \cite{AAPT}.

Our first result in this paper, Theorem \ref{BD4_1} (using different method for the proof) improves and generalizes the non-decreasing cases of Theorem \ref{MS} and Theorem \ref{ATT} and improves the non-increasing case of Theorem \ref{BNthm} . (The non-decreasing sequence $\{q_{k}\}$ corresponds to the non-increasing sequence $\{t_{k,n}\}$ with connection \eqref{conn}.)

\section{Auxiliary results}	

\begin{lemma}[Paley's lemma \cite{SWSP}, p. 7.]\label{Pl}
	Let	$n\in \mathbb{N}$. Then
	\begin{align*}
		D_{2^{n}}(x)=
		\begin{cases}
			2^{n}, & \textrm{if } x\in I_{n},\\
			0, & \textrm{if } x\notin I_{n}.
		\end{cases}
	\end{align*}
\end{lemma}
The next lemma is a simple corollary of Lemma \ref{Pl} and the definition of Dirichlet kernel functions. It can be found in several articles in the literature. For example see \cite{AG}.
\begin{lemma}\label{Dbont}
	Let $k,n\in\mathbb{N}$ and $k<2^{n}$. Then we have 
	\[
	D_{2^{n}-k}=D_{2^{n}}-w_{2^{n}-1}D_{k}.
	\]
\end{lemma}
\begin{lemma}[G\'at \cite{G1}, Corollary 6]\label{K2n}
Let	$n,t\in \mathbb{N}$ and $t<n$. Then
\begin{align*}
	K_{2^{n}}(x)=
	\begin{cases}
		2^{t-1}, & \textrm{if } x\in I_{t}\backslash I_{t+1},\ x-e_{t}\in I_{n},\\
		\frac{2^{n}+1}{2}, & \textrm{if } x\in I_{n},\\	
		0, & \textrm{otherwise.}
	\end{cases}
\end{align*}
\end{lemma}
\begin{lemma}[Fine \cite{Fine}, Lemma 2]\label{Fine}
	Let $n=2^{k}+m$,  where $k,m\in\mathbb{N}$ and $m\leq 2^{k}$. Then
	\[
		nK_{n}=2^{k}K_{2^{k}}+mD_{2^{k}}+r_{k}mK_{m}.
	\]
\end{lemma}
\begin{lemma}[Yano \cite{Y1}]
	Let $n\in \mathbb{P}$, then
	\[
	\left\|K_{n}\right\|_{1}\leq 2.
	\]\end{lemma}
In 2018, Toledo improved Yano's classical result.
\begin{lemma}[Toledo \cite{Tol}]\label{Tol}  
	\[
	\sup_{n\in\mathbb{P}}\|K_{n}\|_{1}= \frac{17}{15}.
	\]
\end{lemma}
The following lemma is a known consequence of Lemma \ref{Pl}, Lemma \ref{K2n} and Lemma \ref{Fine}.
\begin{lemma}\label{nKn}
Let $n\in\mathbb{P}$. Then we have 
\begin{align*}
	n|K_{n}|\leq3\sum_{l=0}^{|n|}2^{l}K_{2^{l}}.
\end{align*}
\end{lemma}
\begin{lemma}[Blahota \cite{B1}, Lemma 9]
	Let $k,n\in\mathbb{P}$ and $\{t_{k,n}: 1\leq k\leq n\}$ be a finite sequence of real numbers. Then
	\begin{align*}	
		K^{T}_{n}=&D_{2^{|n|}}\sum_{k=1}^{n}t_{k,n}-w_{2^{|n|}-1}t_{1,n}\left(2^{|n|}-1\right)K_{2^{|n|}-1}\\
		&+w_{2^{|n|}-1}\sum_{k=1}^{2^{|n|}-2}\Delta t_{2^{|n|}-k-1,n}kK_{k}+r_{|n|}\sum_{k=1}^{n-2^{|n|}}t_{2^{|n|}+k,n}D_{k}.
	\end{align*}
\end{lemma}
\begin{corollary}\label{felb}
	Let $k\in\mathbb{P},\ n\in\mathbb{N}$ and $\ \{t_{k,2^{n}}: 1\leq k\leq 2^{n}\}$ be a finite sequence of real numbers.
	Then
	\begin{align*}
		K_{2^{n}}^{T}&=		D_{2^{n}}\sum_{k=1}^{2^{n}}t_{k,2^{n}}-w_{2^{n}-1}t_{1,2^{n}}(2^{n}-1)K_{2^{n}-1}+w_{2^{n}-1}\sum_{k=1}^{2^{n}-2}\Delta t_{2^{n}-k-1,2^{n}}kK_{k}
	\end{align*}
	holds.
\end{corollary}
\begin{lemma}[Persson, Tephnadze and Weisz \cite{PTW} p. 137-138.]\label{3Ts}Let $n\in\mathbb{N}$ and $f\in L_{p}(G)$ for some $1\leq p<\infty$. Then we have inequality
	\[
		\left\|\sigma_{n}(f)-f\right\|_{p}\leq3\sum_{s=0}^{|n|}\frac{2^{s}}{2^{|n|}}\omega_{p}\left(f,\frac{1}{2^{s}}\right).
	\]
\end{lemma}
\section{On the rate of the approximation}

\begin{theorem}\label{BD4_1}
	Let $f\in L_{p}(G), 1 \leq p \leq \infty$. For every $n\in\mathbb{N},\ \{t_{k,n}: 1\leq k\leq n\}$ be a finite sequence of non-negative numbers such that 
	$$
	\sum_{k=1}^n t_{k,n}=1
	$$
	is satisfied. If the finite sequence $ \{ t_{k,n}: 1\leq k \leq n \}$  is  non-increasing for a  fixed $n$, then 
	\begin{equation*}
		\left\Vert\sigma^{T}_{n}(f)-f\right\Vert_{p}\leq \frac{31}{15}\sum_{k=0}^{|n|-1}2^{k}t_{2^{k},n}
		\omega_{p}\left(f,\frac{1}{2^{k}}\right)+\frac{47}{30}
		\omega_{p}\left(f,\frac{1}{2^{|n|}}\right)
	\end{equation*}
	holds.
\end{theorem}
\begin{proof}
	In our paper \cite{BD3} we proved a similar theorem for Walsh-Kaczmarz system. Based on that proof, using the same method with Lemma \ref{Tol} we get our result easily.
\end{proof}

Theorem \ref{BD4_2} is a generalized (from N\"orlund and $T$ to matrix transform summation) and improved version of Theorem 2 from \cite{AT}.

\begin{theorem}\label{BD4_2}
	Let the finite sequence $\ \{t_{k,2^{n}}: 1\leq k\leq 2^{n}\}$ of non-negative numbers be non-decreasing for all $n\in\mathbb{N}$ and \begin{equation}\label{sum1}
		\sum_{k=1}^{2^{n}}t_{k,2^{n}}=1.
	\end{equation}
	Then for any $f\in L_{p}(G)$ for some $1\leq p<\infty$, we have the following inequality
	\begin{align*}
		\left\|\sigma^{T}_{2^{n}}(f)-f\right\|_{p}\leq& \sum_{s=0}^{n-1}\frac{2^{s}}{2^{n}}\omega_{p}\left(f,\frac{1}{2^{s}}\right)+3\sum_{s=0}^{n-1}(n-s)2^{s}t_{2^{n}-2^{s}+1,n}\omega_{p}\left(f,\frac{1}{2^{s}}\right)\\
		&+\left(2+\frac{1}{2^{n}}\right)\omega_{p}\left(f,\frac{1}{2^{n}}\right).
	\end{align*}
\end{theorem}
\begin{proof}
	Using Corollary \ref{felb}, condition \eqref{sum1} and equality
	\[
		\left(2^{n}-1\right)K_{2^{n}-1}=2^{n}K_{2^{n}}-D_{2^{n}},
	\]
	we obtain
	\begin{align*}
		\sigma^{T}_{2^{n}}(f;x)-&f(x)\\
		=&\int_{G}D_{2^{n}}(u)(f(x+u)-f(x))d\mu(u)\\&-\int_{G}w_{2^{n}-1}(u)t_{1,2^{n}}2^{n}K_{2^{n}}(u)(f(x+u)-f(x))d\mu(u)\\
		&+\int_{G}w_{2^{n}-1}(u)t_{1,2^{n}}D_{2^{n}}(u)(f(x+u)-f(x))d\mu(u)\\&+\int_{G}w_{2^{n}-1}(u)\sum_{k=1}^{2^{n}-2}\Delta t_{2^{n}-k-1,2^{n}}kK_{k}(u)(f(x+u)-f(x))d\mu(u)\\
		=:&I+II+III+IV.
	\end{align*}
	Based on the generalized Minkowski's inequality, Lemma \ref{Pl}, condition \eqref{sum1} and the monotonicity of the sequence $\{t_{k,2^{n}}: 1\leq k\leq 2^{n}\}$, we obtain inequalities
	\begin{align}\label{egy}
		\left\|I\right\|_{p}&\leq\int_{I_{n}}D_{2^{n}}(u)\left\|f(.+u)-f(.)\right\|_{p}d\mu(u)\nonumber\\&\leq \omega_{p}\left(f,\frac{1}{2^{n}}\right)
 	\end{align}
 	and 
 	\begin{align}\label{harom}
 		\left\|III\right\|_{p}&\leq t_{1,2^{n}}\int_{I_{n}}D_{2^{n}}(u)\left\|f(.+u)-f(.)\right\|_{p}d\mu(u)\nonumber\\&\leq  \frac{1}{2^{n}}\omega_{p}\left(f,\frac{1}{2^{n}}\right).
 	\end{align}
  	Combining the  generalized Minkowski's inequality and Lemma \ref{K2n}, we get
   \begin{align}\label{ketto}
 	\left\|II\right\|_{p}
 	\leq&\int_{G}\left\|f(.+u)-f(.)\right\|_{p}K_{2^{n}}(u)d\mu(u)\nonumber\\
 	=&\int_{I_{n}}\left\|f(.+u)-f(.)\right\|_{p}K_{2^{n}}(u)d\mu(u)\nonumber\\
 	&+\sum_{s=0}^{n-1}\int_{I_{n}(e_{s})}\left\|f(.+u)-f(.)\right\|_{p}K_{2^{s}}(u)d\mu(u)\nonumber\\
 	\leq&\frac{2^{n}+1}{2}\int_{I_{n}}\left\|f(.+u)-f(.)\right\|_{p}d\mu(u)\nonumber\\
 	&+\sum_{s=0}^{n-1}2^{s-1}\int_{I_{n}(e_{s})}\left\|f(.+u)-f(.)\right\|_{p}d\mu(u)\nonumber\\ 	\leq&\frac{2^{n}+1}{2}\frac{1}{2^{n}}\omega_{p}\left(f,\frac{1}{2^{n}}\right)+\sum_{s=0}^{n-1}\frac{2^{s-1}}{2^{n}}\omega_{p}\left(f,\frac{1}{2^{s}}\right)\nonumber\\
 	\leq&\omega_{p}\left(f,\frac{1}{2^{n}}\right)+\sum_{s=0}^{n-1}\frac{2^{s}}{2^{n}}\omega_{p}\left(f,\frac{1}{2^{s}}\right).
 \end{align}
 On the other side, it means that inequality
 \begin{align}\label{2n}
	2^n\int_{G}\left\|f(.+u)-f(.)\right\|_{p}K_{2^{n}}(u)d\mu(u)\leq \sum_{s=0}^{n}2^{s}\omega_{p}\left(f,\frac{1}{2^{s}}\right)
  \end{align}
	is true for all $n\in\mathbb{N}$. 
	
	From the generalized Minkowski's inequality, Lemma \ref{nKn} and inequality \eqref{2n}, we get that
 	\begin{align}\label{negy}
 		\left\|IV\right\|_{p}&\leq\sum_{k=0}^{2^{n}-2}\left|\Delta t_{2^{n}-k-1,2^{n}}\right|\int_{G}\left\|f(.+u)-f(.)\right\|_{p}k|K_{k}(u)|d\mu(u)\nonumber\\
 		&\leq 3\sum_{j=0}^{n-1}\sum_{k=2^{j}-1}^{2^{j+1}-2}\left|\Delta t_{2^{n}-k-1,2^{n}}\right|\sum_{l=0}^{j}2^{l}\int_{G}\left\|f(.+u)-f(.)\right\|_{p}K_{2^{l}}(u)d\mu(u)\nonumber\\
 		&\leq 3\sum_{j=0}^{n-1}\sum_{k=2^{j}-1}^{2^{j+1}-2}\left|\Delta t_{2^{n}-k-1,2^{n}}\right|\sum_{l=0}^{j}\sum_{s=0}^{l}2^{s}\omega_{p}\left(f,\frac{1}{2^{s}}\right)\nonumber\\
 		&=3\sum_{j=0}^{n-1}( t_{2^{n}-2^{j}+1,2^{n}}-t_{2^{n}-2^{j+1}+1,2^{n}})\sum_{l=0}^{j}\sum_{s=0}^{l}2^{s}\omega_{p}\left(f,\frac{1}{2^{s}}\right)\nonumber\\
 		&=3\sum_{l=0}^{n-1}\sum_{j=l}^{n-1}( t_{2^{n}-2^{j}+1,2^{n}}-t_{2^{n}-2^{j+1}+1,2^{n}})\sum_{s=0}^{l}2^{s}\omega_{p}\left(f,\frac{1}{2^{s}}\right)\nonumber\\
 		&\leq3\sum_{l=0}^{n-1}t_{2^{n}-2^{l}+1,2^{n}}\sum_{s=0}^{l}2^{s}\omega_{p}\left(f,\frac{1}{2^{s}}\right)\nonumber\\	
 		&=3\sum_{s=0}^{n-1}2^{s}\omega_{p}\left(f,\frac{1}{2^{s}}\right)\sum_{l=s}^{n-1}t_{2^{n}-2^{l}+1,2^{n}}\nonumber\\	
 		&\leq3\sum_{s=0}^{n-1}(n-s)2^{s}t_{2^{n}-2^{s}+1,2^{n}}\omega_{p}\left(f,\frac{1}{2^{s}}\right).
 	\end{align}
	Taking inequalities \eqref{egy}, \eqref{harom}, \eqref{ketto} and \eqref{negy} together, we see that the theorem is proved.
	\end{proof}
	Theorem \ref{BD4_3} is also a generalization; it based on Theorem 3 from \cite{AT}.
	\begin{theorem}\label{BD4_3}
	For every $n\in\mathbb{P}$, let the finite sequence $\ \{t_{k,n}: 1\leq k\leq n\}$ of non-negative numbers be non-decreasing for all $n$ and we suppose that
	\begin{equation}\label{cond1}
		\sum_{k=1}^{n}t_{k,n}=1
	\end{equation}
	and
	\begin{equation}\label{cond3}
		t_{n,n}=O\left(\frac{1}{n}\right).
	\end{equation}
	Then for any $f\in L_{p}(G)$ for some $1\leq p<\infty$, we have the following inequality
	\begin{align*}
		\left\|\sigma^{T}_{n}(f)-f\right\|_{p}\leq c\sum_{k=0}^{|n|}\frac{2^{k}}{2^{|n|}}\omega_{p}\left(\frac{1}{2^{k}},f\right).
	\end{align*}
	\end{theorem}

	\begin{proof}
	From Abel transformation it is easy to get equality
	\begin{equation}\label{Abel1}
		\sum_{k=1}^{n}t_{k,n}=\sum_{k=1}^{n-1}\Delta t_{k,n}k+t_{n,n}n
	\end{equation}	
	and a connection between $T$ and Fej\'er means
	\begin{equation}\label{Abel2}
		\sigma^{T}_{n}(f)=\sum_{k=1}^{n-1}\Delta t_{k,n}k\sigma_{k}(f)+t_{n,n}n\sigma_{n}(f).
	\end{equation}
	Using \eqref{cond1}, \eqref{Abel1} and \eqref{Abel2} we can conclude equality
	\[
	\sigma^{T}_{n}(f)-f=\sum_{k=1}^{n-1}\Delta t_{k,n}k(\sigma_{k}(f)-f)+t_{n,n}n(\sigma_{n}(f)-f)
	\]
	and 
	\begin{align*}
		\left\|\sigma^{T}_{n}(f)-f\right\|_{p}\leq\sum_{k=1}^{n-1}|\Delta t_{k,n}|k\left\|\sigma_{k}(f)-f\right\|_{p}+t_{n,n}n\left\|\sigma_{n}(f)-f\right\|_{p}=:I+II.
	\end{align*}
	Then
	\[
		I=\sum_{k=1}^{2^{|n|}-1}|\Delta t_{k,n}|k\left\|\sigma_{k}(f)-f\right\|_{p}+\sum_{k=2^{|n|}}^{n-1}|\Delta t_{k,n}|k\left\|\sigma_{k}(f)-f\right\|_{p}=:I_{1}+I_{2}.	
	\]
	Lemma \ref{3Ts}, condition \eqref{cond3} and $|\Delta t_{k,n}|=-\Delta t_{k,n}=t_{k+1,n}-t_{k,n}$ imply
	\begin{align*}
		I_{1}&=\sum_{l=0}^{|n|-1}\sum_{k=2^{l}}^{2^{l+1}-1}|\Delta t_{k,n}|k\left\|\sigma_{k}(f)-f\right\|_{p}\\
		&\leq\sum_{l=0}^{|n|-1}\sum_{k=2^{l}}^{2^{l+1}-1}|\Delta t_{k,n}|k3\sum_{s=0}^{l}\frac{2^{s}}{2^{l}}\omega_{p}\left(f,\frac{1}{2^{s}}\right)\\
		&\leq3\sum_{l=0}^{|n|-1}\frac{2^{l+1}}{2^{l}}\sum_{k=2^{l}}^{2^{l+1}-1}|\Delta t_{k,n}|\sum_{s=0}^{l}{2^{s}}\omega_{p}\left(f,\frac{1}{2^{s}}\right)\\
		&=6\sum_{l=0}^{|n|-1}(t_{2^{l+1},n}-t_{2^{l},n})\sum_{s=0}^{l}{2^{s}}\omega_{p}\left(f,\frac{1}{2^{s}}\right)\\
		&=6\sum_{s=0}^{|n|-1}{2^{s}}\omega_{p}\left(f,\frac{1}{2^{s}}\right)\sum_{l=s}^{|n|-1}(t_{2^{l+1},n}-t_{2^{l},n}
	\end{align*}		
	\begin{align*}	
		&=6\sum_{s=0}^{|n|-1}{2^{s}}\omega_{p}\left(f,\frac{1}{2^{s}}\right)(t_{2^{|n|},n}-t_{2^{s},n})\\
		&\leq6t_{n,n}\sum_{s=0}^{|n|-1}{2^{s}}\omega_{p}\left(f,\frac{1}{2^{s}}\right)\\
		&\leq c\sum_{s=0}^{|n|-1}\frac{2^{s}}{2^{|n|}}\omega_{p}\left(f,\frac{1}{2^{s}}\right).
	\end{align*}
	Since if $2^{|n|}\leq k<n$, then $|k|=|n|$, Lemma \ref{3Ts}, equality \eqref{Abel1} and condition \eqref{cond3} result in that
	\begin{align*}
		I_{2}&\leq\sum_{k=2^{|n|}}^{n-1}|\Delta t_{k,n}|k3\sum_{s=0}^{|k|}\frac{2^{s}}{2^{|k|}}\omega_{p}\left(f,\frac{1}{2^{s}}\right)\\
		&\leq3(nt_{n,n}-1)\sum_{s=0}^{|n|}\frac{2^{s}}{2^{|n|}}\omega_{p}\left(f,\frac{1}{2^{s}}\right)\\
		&\leq c\sum_{s=0}^{|n|}\frac{2^{s}}{2^{|n|}}\omega_{p}\left(f,\frac{1}{2^{s}}\right).
	\end{align*}
	Also similarly
	\begin{align*}
		II&\leq t_{n,n}n3\sum_{s=0}^{|n|}\frac{2^{s}}{2^{|n|}}\omega_{p}\left(f,\frac{1}{2^{s}}\right)\\
		&\leq c\sum_{s=0}^{|n|}\frac{2^{s}}{2^{|n|}}\omega_{p}\left(f,\frac{1}{2^{s}}\right).
	\end{align*}
	This completes the proof of the theorem.  
	\end{proof}

\end{document}